\documentclass[11pt,reqno]{amsart}
\usepackage{amsmath,epsfig,graphicx,color}
\usepackage{ifthen}
\usepackage{dsfont}
\usepackage{amssymb,latexsym}
\usepackage{subfigure}
\usepackage{hyperref}
\hypersetup{
urlcolor=black, 
  menucolor=black, 
  citecolor=black, 
  anchorcolor=black, 
  filecolor=black, 
  linkcolor=black, 
  colorlinks=true,
}
\usepackage{comment}
\usepackage[numbers,sort&compress]{natbib} 

\newcommand{\wt }{\widetilde}

\newcommand{\bfA}{{\mathbf A}}

\newcommand{\bfR}{{\mathbf R}}

\newcommand{\bfM}{{\mathbf M}}
\newcommand{\bfLambda}{{\mathbf \Lambda}}
\newcommand{\bfGamma}{{\mathbf \Gamma}}
\newcommand{\bfSigma}{{\mathbf \Sigma}}
\newcommand{\bfV}{{\mathbf V}}

\newcommand{\bfQ}{{\mathbf Q}}
\newcommand{\bfI}{{\mathbf I}}
\newcommand{\bfX}{{\mathbf X}}
\newcommand{\bfY}{{\mathbf Y}}
\newcommand{\bfx}{{\mathbf x}}

\newcommand{\la}{\lambda}

\newcommand{\beao}{\begin{eqnarray*}}
\newcommand{\eeao}{\end{eqnarray*}}
\newcommand{\beam}{\begin{eqnarray}}
\newcommand{\eeam}{\end{eqnarray}}
\newcommand{\barr}{\begin{array}}
\newcommand{\earr}{\end{array}}

\definecolor{darkblue}{rgb}{.1, 0.1,.8}
\definecolor{darkgreen}{rgb}{0,0.8,0.2}
\definecolor{darkred}{rgb}{.8, .1,.1}
\newcommand{\bco}{\begin{corrolary}}
\newcommand{\eco}{\end{corrolary}}

\newcommand{\E}{\mathbb{E}}
\renewcommand{\P}{\mathbb{P}}
\newcommand{\1}{\mathds{1}}
\newcommand{\R}{\mathbb{R}}

\newcommand{\C}{\mathbb{C}}
\newcommand{\bfC}{{\mathbf C}}
\newcommand{\bfS}{{\mathbf S}}
\newcommand{\bfB}{{\mathbf B}}
\newcommand{\Z}{\mathbb{Z}}
\newcommand{\Var}{\operatorname{Var}}

\newcommand{\Frechet}{Fr\'{e}chet }

\newcommand{\x}{{\mathbf x}}

\newcommand{\bfz}{{\mathbf z}}
\newcommand{\X}{{\mathbf X}}
\newcommand{\Y}{{\mathbf Y}}

\newcommand{\A}{{\mathbf A}}
\newcommand{\bfZ}{{\mathbf Z}}
\newcommand{\z}{{\mathbf Z}}

\newcommand{\dint}{\,\mathrm{d}}

\newcommand{\twonorm}[1]{\|#1\|}

\newcommand{\vep}{\varepsilon}
\newcommand{\nto}{n \to \infty}
\newcommand{\pto}{p \to \infty}

\newcommand{\lhs}{left-hand side}

\newcommand{\tr}{\operatorname{tr}}
\newcommand{\spec}{\operatorname{spec}}
\newcommand{\supp}{\operatorname{supp}}

\newcommand{\diag}{\operatorname{diag}}
\newcommand{\offdiag}{\operatorname{offdiag}}

\newcommand{\MP}{Mar\v cenko--Pastur }

\def\tag{\refstepcounter{equation}\leqno }

\newtheorem{lemma}{Lemma}[section]

\newtheorem{theorem}[lemma]{Theorem}

\newtheorem{example}[lemma]{Example}

\newtheorem{remark}[lemma]{Remark}

\newcommand{\cas}{\stackrel{\rm a.s.}{\longrightarrow}}

\newcommand{\eid}{\stackrel{d}{=}}

\newcommand{\as}{{\rm a.s.}}

\begin{document}
\bibliographystyle{acm}
\title[Large sample correlation matrices]
{Large sample correlation matrices: a comparison theorem and its applications}
\author[J. Heiny]{Johannes Heiny}
\address{Fakult\"at f\"ur Mathematik,
Ruhruniversit\"at Bochum,
Universit\"atsstrasse 150,
D-44801 Bochum,
Germany}
\email{johannes.heiny@rub.de}

\begin{abstract}
In this paper, we show that the diagonal of a high-dimensional sample covariance matrix stemming from $n$ independent observations of a $p$-dimensional time series with finite fourth moments can be approximated in spectral norm by the diagonal of the population covariance matrix. We assume that $n,p\to \infty$ with $p/n$ tending to a constant which might be positive or zero. As applications, we provide an approximation of the sample correlation matrix $\bfR$ and derive a variety of results for its eigenvalues. We identify the limiting spectral distribution of $\bfR$ and construct an estimator for the population correlation matrix and its eigenvalues. Finally, the almost sure limits of the extreme eigenvalues of $\bfR$ in a generalized spiked correlation model are analyzed.
\end{abstract}

\maketitle

\section{Introduction}\setcounter{equation}{0}

In time series analysis the notions of covariance and correlation 
play a vital role in multivariate statistical analysis for  parameter estimation, goodness-of-fit tests, change-point detection, etc.; see for example the classical monographs \cite{brockwell:davis:1991,priestley:1981}. 
With the rapid improvements of modern data collection devices, large data sets occur in many fields such as finance, telecommunications and meteorology.   
In high-dimensional data analyzes, a good understanding of the population correlation and covariance matrices provides important insight into the dependence structure and the geometry of the underlying distribution.
When considering random matrices $\X=\X_n= (\x_1,\ldots, \x_n)$
with high-dimensional time series observations $\x_t= (X_{1t},\ldots,X_{pt})'$, $t\in \Z$, the main focus of interest 
has been on the asymptotic
properties of the eigenvalues of the sample covariance matrix $\bfS=n^{-1} \X\X'$;
see for instance \cite{bai:silverstein:2010,yao:zheng:bai:2015}. A large amount of recent literature is devoted to the study of $\bfS$ in a setting where the dimension $p$ and and the sample size $n$ are of comparable magnitude, that is, the ratio $p/n$ tends to a positive constant as $n,p\to \infty$; see \cite{johnstone:2001,donoho:2000} for a discussion of typical applications where such an assumption is natural.

 Under finite variance of the entries of $\X$,
  the spectral properties of the sample covariance matrix $\bfS$
  have been well analyzed in random matrix theory since the pioneering work
  \cite{marchenko:pastur:1967} where it is shown that for iid (independent and identically distributed) components of $\x_t$ the empirical spectral distribution of $\bfS$ converges weakly to the
  celebrated \MP law. For many time series models the limiting spectral distribution can be characterized in terms of an integral equation for its Stieltjes transform. Explicit solutions are more involved; see the monographs \cite{bai:silverstein:2010,bai:fang:liang:2014,yao:zheng:bai:2015}. Over the last couple of years significant progress on limiting spectral distributions for dependent time series was achieved \cite{banna:merlevede:peligrad:2015,banna:merlevede:2015,banna:2016}. Subsequently, several  ground-breaking results such as  the
  convergence of the largest eigenvalue $\lambda_1(\bf S)$ and the smallest eigenvalue $\lambda_p(\bfS)$ to the edges of  the \MP law
  \citep{BaiYin88a,tikhomirov:2015},
  asymptotic  normality of linear spectral statistics of $\bfS$ \citep{BS04},  or its edge
  universality towards the Tracy-Widom law \citep{johnstone:2001,Peche2012,PillaiYin2014} were established. 
	Apart from the convergence of $\lambda_p(\bfS)$ all those results require a finite fourth moment of the entries of $\X$.
	
	In case of infinite fourth moments, the theory for the eigenvalues and eigenvectors of $\bfS$ is quite different from the aforementioned \MP~theory. For example, if the distribution of the $X_{ij}$ is regularly varying with index $\alpha\in (0,4)$, the properly normalized largest eigenvalue of $\bfS$ converges to a \Frechet distribution with parameter $\alpha/2$. A detailed account on the developments in the heavy-tailed case can be found in \cite{davis:heiny:mikosch:xie:2016,heiny:mikosch:2017:iid,basrak:heiny:jung:2020,auffinger:arous:peche:2009}.

	
	For the sample correlation matrix $\bfR=\{\diag(\bfS)\}^{-1/2} \bfS \{\diag(\bfS)\}^{-1/2}$, the situation gets more complicated because of the specific nonlinear dependence structure caused by the normalization $\{\diag(\bfS)\}^{-1/2}$,  which makes the analysis of this random matrix quite challenging. As a consequence, the study of the high-dimensional sample correlation matrix is more recent and somewhat limited. Sometimes practitioners would like to know ``to which extent the random matrix results would hold if one were concerned with sample correlation matrices and not sample covariance matrices \cite{elkaroui:2009}''. In case the elements of the data matrix $\X$ are iid with zero mean, variance equal to one and finite fourth moment it is shown by Jiang \cite{jiang:2004} (see also \cite{elkaroui:2009} and \cite{heiny:mikosch:2017:corr}) that the \MP~law is still valid for the sample correlation matrix $\bfR$. The first result for the linear spectral statistics of $\bfR$ was proved in \cite{Gao2017} under existence of the fourth moment. Moreover, the properly normalized largest off-diagonal entry of $\bfR$ converges to a Gumbel distribution as shown in \cite{jiang:2004b} and recently generalized to a point process setting in \cite{heiny:mikosch:yslas:2021}.	
	
The aim of this paper is to provide asymptotic theory for the sample correlation matrix and to estimate the population correlation matrix. We will show that, under a finite fourth moment assumption (that cannot be improved), the diagonal matrix $\{\diag(\bfS)\}^{-1/2}$ is approximated in spectral norm by the diagonal of the population covariance matrix, which in turn greatly simplifies the study of $\bfR$.

\subsection{The model and notation}
We consider $n$ independent and identically distributed (iid) observations $\x_t$ from a $p$-dimensional time series. 
The observations have the structure $\x_t=\bfA \bfz_t$, where 
\begin{equation*}
\bfz_t=(Z_{1t},Z_{2t}, \ldots,Z_{pt})'\,, \quad t=1,\ldots,n\,,
\end{equation*}
are iid random vectors with iid real-valued components
with generic element $Z\eid Z_{it}$. Throughout this paper, we assume that $\E[Z]=0$ and $\E[Z^2]=1$, unless explicitly stated otherwise.  
Moreover, $\A=\A_p \in \R^{p\times p}$
is a sequence of $p\times p$ matrices satisfying
\begin{equation}\label{assumptionA}
c_1 <  \min_{i=1,\ldots,p} (\bfA_p \bfA_p')_{ii} \le  \twonorm{\bfA_p}^2 \le c_2 \,, \qquad p\ge 1\,,
\end{equation}
for some positive constants $c_1,c_2$. {\em Note that the subscript $p$ is usually suppressed in our notation. For any $p\times p$ matrix $\bfC$ with real eigenvalues, we denote its ordered eigenvalues by 
$\la_1(\bfC) \ge \cdots \ge \la_p(\bfC)$. The spectral (or operator) norm of $\bfC$ is given as $\twonorm{\bfC}=\sqrt{\la_1(\bfC \bfC')}$, while the {\em empirical spectral distribution} of $\bfC$ is defined by 
$F_{\bfC}(x)= \frac{1}{p}\; \sum_{i=1}^p \1_{\{ \lambda_i(\bfC)\le x \}}$ for $ x\in  \R$.}

Setting $\z=(\bfz_1,\ldots, \bfz_n)=(Z_{it})_{i=1,\ldots,p; t=1,\ldots,n}$ our data matrix $\X$ becomes
\begin{equation}\label{eq:model}
\X=\A \z\,
\end{equation}
and the $p\times p$ {\em sample covariance matrix} $\bfS=(S_{ij})$ is given by
\begin{equation}\label{eq:scoooo}
\bfS=\frac{1}{n} \sum_{t=1}^n \x_t \x_t'=\frac{1}{n}\X\X'\,.
\end{equation}
Note that the columns of $\X$ are independent.
The so-called {\em population covariance matrix} $\bfSigma$, which is defined as the covariance matrix of the first column of $\X$, equals $\bfSigma=\A\A'$.
Since $\X$ is centered, we have $\E[\bfS]= \bfSigma$.

\begin{remark}{\em The uniform boundedness of $(\twonorm{\bfA})$ prevents an explosion of the largest eigenvalue of $\bfS$ caused by the deterministic matrix $\bfA$. Indeed, we have $\twonorm{\bfS}\le \twonorm{\bfA}^2 \twonorm{n^{-1} \z\z'}$. The lower bound in \eqref{assumptionA} ensures that the variance of each component of the observations does not vanish. This way we avoid asymptotically non-random components of $\bfx_t$.
}\end{remark}

While the literature on large sample covariance matrices $\bfS$ is extensive, the {\em sample correlation matrix} $\bfR=(R_{ij})=\{\diag(\bfS)\}^{-1/2} \bfS \{\diag(\bfS)\}^{-1/2}$ with entries 
\begin{equation}\label{eq:corrR}
R_{ij}=\frac{S_{ij}}{\sqrt{S_{ii} S_{jj}}}\,, \qquad i,j=1,\ldots,p\,,
\end{equation}
has been less studied. {\em Here, $\diag(\bfS)$ denotes the diagonal matrix with the same diagonal elements as $\bfS$. Sometimes we will simply refer to $\diag(\bfS)$ as the diagonal of $\bfS$. Analogously, we define $\offdiag(\bfS)=\bfS-\diag(\bfS)$.}

Next, we introduce the {\em population correlation matrix} $\bfGamma=(\Gamma_{ij})$
which takes the form 
\begin{equation}\label{eq:empconnection}
\bfGamma =(\diag(\bfSigma))^{-1/2} \bfSigma (\diag(\bfSigma))^{-1/2} \,.
\end{equation}
For $1\le i,j \le p$, $R_{ij}$ is an asymptotically unbiased estimator of $\Gamma_{ij}$. More precisely,
\begin{equation}\label{eq:exprij}
\E[R_{ij}]=\Gamma_{ij}+O(n^{-1})\,, \quad \text{ as } \pto\,;
\end{equation}
see for instance \cite[eq.~(4)]{lai:rayner:hutchinson:1999}.


\subsubsection{Growth rates}
Both the dimension $p$ and the sample size $n$ are large and tend to infinity together.
In this paper, the sample size is a function of the dimension and the dimension increases at most proportionally to the sample size. To be precise, we assume
\begin{equation}\label{Cgamma}
n=n_p \to \infty \quad \text{ and } \quad \frac{p}{n_p}\to \gamma\in [0,\infty)\,,\quad \text{ as } \pto\,. \tag{$C_{p,n}$}
\end{equation}
The constant $\gamma$ controls the growth of the dimension relative to the sample size. Most of the random matrix literature focuses exclusively on the case $\gamma>0$, while the case $\gamma=0$ plays only a minor role. In many fields, however, the wider range of possible growth rates arising in the $\gamma=0$ regime is desirable. The framework in this paper unifies these two lines of research.

\subsection{First result: approximation of the diagonal of the sample covariance matrix}

We provide an approximation of $\diag(\bfS)$.
The next result shows that, despite the dependence within the columns of $\bfX$, the diagonal of the sample covariance matrix can be approximated by $\diag(\bfSigma)$ and the quality of the approximation improves if the order of the ratio $n/p$ increases.

\begin{theorem}\label{thm:main}
We consider the model $\X=\bfA\bfZ$ from \eqref{eq:model}, where the matrix $\bfA$ satisfies \eqref{assumptionA} and $\z$ has iid entries. Assume the growth condition \eqref{Cgamma} and $\E[Z^4]< \infty$.  Then we have
\begin{equation}\label{eq:main}
\sqrt{n/p} \,\twonorm{\diag(\bfS)-\diag(\bfSigma)} \cas 0\,, \qquad p \to \infty\,,
\end{equation}
and
\begin{equation}\label{eq:mainsqrt}
\sqrt{n/p} \,\twonorm{(\diag(\bfS))^{-1/2}-(\diag(\bfSigma))^{-1/2}} \cas 0\,, \qquad p \to \infty\,.
\end{equation}
\end{theorem}
The proof will be presented in Section~\ref{sec:proofs}.
Theorem~\ref{thm:main} is the key to deriving a multitude of asymptotic results for the sample correlation matrix $\bfR$,  
\begin{equation}\label{eq:RS}
\bfR= (\diag(\bfS))^{-1/2} \bfS (\diag(\bfS))^{-1/2}\,.
\end{equation} 
We have $\bfR= \Y\Y'$, where 
\begin{equation*}
\Y=n^{-1/2}\,(\diag(\bfS))^{-1/2} \X \quad \text{ and } \quad Y_{ij}=\frac{X_{ij}}{\sqrt{\sum_{\ell=1}^n X_{i\ell}^2}}\,.
\end{equation*}
In general, any two entries of $\bfY=(Y_{ij})$ are dependent. This is in stark contrast to \eqref{eq:scoooo} because the data matrix $\X$ possesses independent columns.
The full dependence within $\Y$ requires a more careful analysis and considerably complicates the proofs of results about $\bfR=\Y\Y'$. Even worse, the dependence caused by multiplication with $(\diag(\bfS))^{-1/2}$ is nonlinear and moment calculations of self-normalized random variables like $Y_{ij}$ are not available in the literature. In this light, Theorem~\ref{thm:main} facilitates the derivation of limit theory for sample correlation matrices immensely. By replacing the stochastic $\diag(\bfS)$ with the deterministic $\diag(\bfSigma)$ one removes the dependence of the columns and the dependence within a column is linear.


\begin{remark}{\em
(1) It is important to note that the moment condition $\E[Z^4]< \infty$ in Theorem~\ref{thm:main} cannot be improved. In fact, in the special case $\bfA=\bfI$ and $p/n\to \gamma>0$, the limit relation \eqref{eq:main} is equivalent to $\E[Z^4]< \infty$, by Lemma~\ref{lem:2}.

(2) If we assume that $|Z|$ is regularly varying with index $\alpha\in (0,4)$ implying $\E[Z^4]= \infty$, then the precise behavior of $\twonorm{\diag(\bfS)-\bfI}$ can be deduced from \cite{heiny:mikosch:2017:iid}. Thus, let $\P(|Z|>x)=x^{-\alpha} L(x)$ for $x>0$, where $L$ is a slowly varying function (at infinity). Then an application of Lemma~3.8 in \cite{heiny:mikosch:2017:iid} yields that $n(np)^{-2/\alpha} \ell(np) \twonorm{\diag(\bfS)-\bfI}$ converges to a \Frechet distributed random variable $\eta$ with parameter $\alpha/2$, as $\pto$, for some slowly varying function $\ell$. Since $\alpha\in (0,4)$, we have
$$\lim_{\pto} \frac{\sqrt{n/p}}{n(np)^{-2/\alpha} \ell(np)}=\infty\,,$$
from which it is easy to conclude that $\sqrt{n/p} \,\twonorm{\diag(\bfS)-\bfI}\to \infty$, as $\pto$.
}\end{remark}

The rest of this paper is structured as follows. In Section~\ref{sec:applications}, Theorem~\ref{thm:main} will be crucial in
\begin{itemize}
\item approximating the sample correlation matrix $\bfR$,
	\item identifying the limiting spectral distribution of $\bfR$,
	\item determining under which growth rates $\bfR$ is a consistent estimator for the population correlation matrix $\bfGamma$ and constructing such an estimator if straightforward estimation is biased,
	\item estimating the population spectrum $(\la_1(\bfGamma),\ldots,\la_p(\bfGamma)$,
	\item finding the \as~limits of sample eigenvalues in a generalized spiked correlation model.
\end{itemize}
The proofs are collected in Section~\ref{sec:proofs}, while
Appendix \ref{sec:appendix} contains some useful auxiliary results.

\section{Applications to sample correlation matrices}\label{sec:applications}\setcounter{equation}{0}

From \eqref{eq:empconnection} we know the connection between the population correlation matrix $\bfGamma$ and the population covariance matrix $\bfSigma$.
An important question is how their empirical versions $\bfR$ and $\bfS$ are related.

\subsection{Approximation of the sample correlation matrix}
In view of Theorem~\ref{thm:main}, it is natural to expect that asymptotically $\diag(\bfS)$ can be replaced by $\diag(\bfSigma)$ in \eqref{eq:RS}.
\begin{theorem}\label{thm:comp}
We consider the model $\X=\bfA\bfZ$ from \eqref{eq:model}, where the matrix $\bfA$ satisfies \eqref{assumptionA} and $\z$ has iid entries. Assume the growth condition \eqref{Cgamma} and $\E[Z^4]< \infty$. Then we have, as $\pto$,
\begin{equation}\label{eq:compeig0}
\sqrt{\frac{n}{p}} \, \twonorm{\bfR-(\diag(\bfSigma))^{-1/2} \bfS(\diag(\bfSigma))^{-1/2}} \cas 0\,.
\end{equation}
\end{theorem}
\begin{proof}
By the triangle inequality, one has
\begin{equation*}
\begin{split}
&\sqrt{\frac{n}{p}}\, \twonorm{\bfR-(\diag(\bfSigma))^{-1/2} \bfS(\diag(\bfSigma))^{-1/2}} \\ 
&= \sqrt{\frac{n}{p}}\, \twonorm{(\diag(\bfS))^{-1/2} \bfS (\diag(\bfS))^{-1/2}-(\diag(\bfSigma))^{-1/2}\bfS(\diag(\bfSigma))^{-1/2}}\\
&\le \sqrt{\frac{n}{p}}\, \twonorm{(\diag(\bfS))^{-1/2}-(\diag(\bfSigma))^{-1/2}} \twonorm{\bfS} (\twonorm{(\diag(\bfS))^{-1/2}}+\twonorm{(\diag(\bfSigma))^{-1/2}})\cas 0\,.
\end{split}
\end{equation*}
Here we used Theorem~\ref{thm:main} and the fact that $\twonorm{\bfS}\le \twonorm{n^{-1} \z\z'} \twonorm{\A\A'}$ is bounded by a constant for sufficiently large $p$.
\end{proof}

Consider the transformed data matrix $\bfQ=(\diag(\bfSigma))^{-1/2} \bfX$ and the associated sample covariance matrix $\bfS^{\bfQ}=n^{-1} \bfQ\bfQ'$. Then \eqref{eq:compeig0} reads as 
\begin{equation}\label{eq:compeig3}
\sqrt{\dfrac{n}{p}}\, \twonorm{\bfR-\bfS^{\bfQ}} \cas 0\,,
\end{equation}
from which we see that the sample correlation matrix is close to the matrix $\bfS^{\bfQ}$.

Since correlations are scale invariant one can always renormalize the data first to ensure that the empirical variance in each component is 1. If $\diag(\bfSigma)=\bfI$, we obtain  
$\sqrt{n/p}\, \twonorm{\bfR-\bfS} \cas 0$
as a special case. As a consequence, $\bfR$ and $\bfS$ possess the same spectral properties.

Thanks to Theorem \ref{thm:comp} many interesting results about the spectrum of $\bfR$ can be directly deduced from the theory of large sample covariance matrices. As examples we will present the limiting spectral distribution of $\bfR$ and the behavior of a variety of eigenvalues of $\bfR$.

As regards the eigenvalues, an application of Weyl's perturbation inequality yields 
\begin{equation}\label{eq:compeig1}
\sqrt{\dfrac{n}{p}} \max_{i=1,\ldots,p} \Big|\la_i(\bfR)-\la_i(\bfS^{\bfQ}) \Big| \le \sqrt{\dfrac{n}{p}}\, \twonorm{\bfR-\bfS^{\bfQ}} \cas 0\,.
\end{equation}


\subsection{Limiting spectral distribution}

A major problem in random matrix theory is to find the weak limit of a sequence of empirical spectral distributions. By almost sure (\as) weak convergence of the sequence of empirical spectral distributions $(F_{\bfR_p})$ 
to a probability distribution $F$, we mean $\lim_{\pto} F_{\bfR_p}(x)=F(x) \,\as$ for all continuity points of $F$.
In this context a useful tool is the {\em Stieltjes transform}
of $F_{\bfR}$:
\begin{equation*}
s_{\bfR}(z)= \int_{\R} \frac{1}{x-z} \dint F_{\bfR}(x) = \frac{1}{p} \tr(\bfR -z \bfI)^{-1}\,, \quad z\in \C^+\,,
\end{equation*}
where $\C^+$ denotes the complex numbers with positive imaginary part. 
Almost sure weak convergence of $(F_{\bfR})$ to $F$ is equivalent to $s_{F_{\bfR}}(z) \to s_F(z)$ a.s. for all $z\in \C^+$.

Our approximation of the sample correlation matrix $\bfR$ also reveals its limiting spectral distribution in a straightforward way. By \cite[Theorem A.45]{bai:silverstein:2010}, the L\'evy distance between the empirical spectral distributions of $\sqrt{n/p} \,\bfR$ and $\sqrt{n/p} \,(\diag(\bfSigma))^{-1/2} \bfS(\diag(\bfSigma))^{-1/2}$ is bounded by the \lhs~in \eqref{eq:compeig0}. 
This observation combined with the limit theory for empirical spectral distributions of sample covariance matrices; see \cite[Theorem~2.14]{yao:zheng:bai:2015} and \cite[Theorem~1]{pan:gao:2012};
yields the following result.
\begin{theorem}\label{thm:lsdR}
Assume the conditions of Theorem~\ref{thm:main} and that the empirical spectral distribution of 
\begin{equation*}
\bfGamma =(\diag(\bfSigma))^{-1/2} \bfSigma (\diag(\bfSigma))^{-1/2} \,
\end{equation*}
(or equivalently $\bfSigma (\diag(\bfSigma))^{-1}$) converges to a probability distribution $H(\cdot)$.
\begin{enumerate}
\item If $p/n \to \gamma\in (0,\infty)$, then $F_{\bfR}$ converges weakly, with probability one, to a unique distribution function $F_{\gamma,H}$, whose Stieltjes transform $s$ satisfies 
\begin{equation}\label{eq:sdgdsgfdg}
s(z)=\int \frac{\dint\! H(t)}{t(1-\gamma -\gamma s(z))-z}\,, \quad z\in \C^+\,.
\end{equation} 
\item If $p/n \to 0$, then $F_{\sqrt{n/p}(\bfR- \bfGamma)}$ converges weakly, with probability one, to a unique distribution function $F$, whose Stieltjes transform $s$ satisfies
\begin{equation}\label{eq:sdgdsgfdg1}
s(z)=-\int \frac{\dint\! H(t)}{z+t \wt s(z)}\,, \quad z\in \C^+\,,
\end{equation} 
where $\wt s$ is the unique solution to $\wt s(z)=-\int (z+t \wt s(z))^{-1} t \dint\! H(t)$ and $z\in \C^+$.
\end{enumerate}
\end{theorem}

Theorem \ref{thm:lsdR} is interesting for applications since it allows a wide range of dependence structures. Part (1) improves Theorem~1 in \cite{elkaroui:2009} where $\E[Z^4 (\log|Z|)^2]<\infty$ was required. In the literature, \eqref{eq:sdgdsgfdg} is sometimes written as 
\begin{equation*}
z=-\frac{1}{\underline{s}(z)}+ \gamma \int \frac{t \dint\! H(t)}{1+t \underline{s}(z)}\,, \quad z\in \C^+\,,
\end{equation*}
where $\underline{s}(z)=- (1-\gamma)/z +\gamma s(z)$.

\begin{example}{\em
We investigate the special case $\bfA=\bfI$ in Theorem~\ref{thm:lsdR}. The empirical spectral distribution of $\bfA$ is the Dirac measure at $1$ and hence $H=\delta_{\{1\}}$.
\begin{enumerate}
	\item If  $p/n \to \gamma\in (0,\infty)$, equation \eqref{eq:sdgdsgfdg} reduces to 
\begin{equation*}
s(z)=\frac{1}{1-\gamma -\gamma s(z)-z}\,, \quad z\in \C^+\,,
\end{equation*}
with solution  
\beam\label{eq:stieltjestransform}
s_{F_{\gamma}}(z)
= \frac{1-\gamma -z +\sqrt{(1+\gamma-z)^2-4\gamma}}{2 \gamma z} \,.
\eeam
This is the Stieltjes transform of the famous Mar\v cenko--Pastur law $F_\gamma$. If $\gamma \in (0,1]$,  $F_\gamma$  has density,
\begin{eqnarray}\label{eq:MPch1}
f_\gamma(x) =
\left\{\begin{array}{ll}
\frac{1}{2\pi x\gamma} \sqrt{(b-x)(x-a)} \,, & \mbox{if } a\le x \le b, \\
 0 \,, & \mbox{otherwise,}
\end{array}\right.
\end{eqnarray}\noindent
where $a=(1-\sqrt{\gamma})^2$ and $b=(1+\sqrt{\gamma})^2$. If $\gamma>1$, the \MP law has an additional point mass $1-1/\gamma$ at $0$.
\item Next, we assume $p/n\to 0$. From \eqref{eq:sdgdsgfdg1} we obtain
$s(z)=- (z+ s(z))^{-1}\,, z\in \C^+$, with solution
\begin{equation*}
s_G(z)=\frac{\sqrt{z^2-4} -z}{2}\,.
\end{equation*}
$s_G$ is the Stieltjes transform of the semicircular law whose density is given by
\begin{equation*}
g(x)=\tfrac{1}{2\pi} \sqrt{4-x^2} \1_{[-2,2]}(x)\,, \quad x\in \R\,.
\end{equation*}
\end{enumerate}
}\end{example}

\subsection{Extreme eigenvalues}

We determine the almost sure limits of the largest and smallest eigenvalues of $\bfR$ and $\bfS$ in the case $\bfA=\bfI$. Since the rank of $\bfS$ is at most $\min(p,n)$ we have $\la_{\min(p,n)+1}(\bfS)=0$. Therefore we interpret $\la_{\min(p,n)}(\bfS)$ as the smallest eigenvalue of $\bfS$. 
\begin{theorem}\label{thm:exiid}
We consider the iid case $\X=\bfZ$. Assume the growth condition \eqref{Cgamma} and $\E[Z^4]< \infty$.
Then 
\begin{equation}\label{eq:extrcorr0}
\lim_{p \to \infty} \sqrt{n/p}\,  (\la_1(\bfR)-1) =2 + \sqrt{\gamma} \quad \text{ and } \quad
\lim_{p \to \infty} \sqrt{n/p} \, (\la_{\min(p,n)}(\bfR)-1) =-2+ \sqrt{\gamma} \quad  \as
\end{equation}
as well as
\begin{equation}\label{eq:extrcorr1}
\lim_{p \to \infty} \sqrt{n/p}\,  (\la_1(\bfS)-1) =2 + \sqrt{\gamma} \quad \text{ and } \quad
\lim_{p \to \infty} \sqrt{n/p} \, (\la_{\min(p,n)}(\bfS)-1) =-2+ \sqrt{\gamma} \quad  \as
\end{equation}
\end{theorem}
A nice feature of this result is that it includes both cases $\gamma>0$ and  $\gamma=0$ which are usually separated in the random matrix literature. The proof of Theorem~\ref{thm:exiid} is provided in Section~\ref{sec:proofs}. 

The novelty of Theorem~\ref{thm:exiid} lies in the case $\gamma=0$, except the limit of $\la_1(\bfS)$ which was shown in \cite[Theorem 3]{chen:pan:2012}.  
For $\gamma>0$  equivalent statements to \eqref{eq:extrcorr0} and \eqref{eq:extrcorr1} were first proved in \cite{bai:yin:krishnaiah:1988,bai:yin:1993,jiang:2004}. In this case Theorem~\ref{thm:exiid} asserts that the largest and smallest eigenvalues converge to the right and left endpoints, respectively, of the \MP law $F_{\gamma}$; see \eqref{eq:MPch1}. For $p/n\to 0$, the extreme eigenvalues tend to one, but after rescaling and centering they converge to the endpoints of the semicircular law.

\subsection{Operator norm consistent estimation of sample correlation matrices}
Under the assumptions of Theorem~\ref{thm:exiid} it follows from \eqref{eq:extrcorr0} that 
\begin{equation}\label{eq:sdsdsdsd}
\sqrt{n/p}\, \twonorm{\bfR-\bfI} \cas 2 + \sqrt{\gamma}\,, \qquad \pto\,.
\end{equation}
Such a result is quite informative regarding the operator norm consistent estimation of sample correlation matrices and will be generalized in the next theorem to general population correlation matrices.

\begin{theorem}\label{thm:est}
Under the assumptions of Theorem~\ref{thm:main} we have, as $\pto$,
\begin{equation*}
\twonorm{\bfR- \bfGamma}= O(\sqrt{p/n})\,\quad \as
\end{equation*}
\end{theorem}
\begin{proof}
By assumption, it holds $\twonorm{\bfA}\le c_2$ and $\min_i \bfSigma_{ii}>c_1$ for some positive constants $c_1,c_2$, where we recall that $\bfSigma=\bfA\bfA'$.
We have 
\begin{equation*}
\twonorm{\bfR- \bfGamma}= \twonorm{ (\diag(\bfS))^{-1/2} \bfS (\diag(\bfS))^{-1/2} -(\diag(\bfSigma))^{-1/2}  \bfSigma (\diag(\bfSigma))^{-1/2}   }\,.
\end{equation*}
Using the triangle inequality to replace $\diag(\bfSigma)$ by $\diag(\bfS)$, an application of Theorem~\ref{thm:main} yields
\begin{equation*}
\sqrt{n/p}\,\twonorm{\bfR- \bfGamma}\le c \sqrt{n/p}\, \twonorm{\bfS-\bfSigma} \le c \sqrt{n/p}\, \twonorm{ n^{-1} \z\z'-\bfI} \twonorm{\bfSigma} = O(1)\quad \as\,,
\end{equation*}
where Theorem~\ref{thm:exiid} was used in the last step. {\em Here and throughout this paper, $c$ stands for some positive constant whose value is not important and may change from line to line.}
\end{proof}

Hence, operator norm consistent estimation is only possible if $p/n\to 0$. Intuitively, this makes a lot of sense. Indeed, it is natural to expect that $\bfR$ constitutes a better estimator for $\bfGamma$ if the sample size $n$ grows at a faster rate than the dimension $p$.

 If $p/n\to\gamma>0$, we have seen in Theorem~\ref{thm:est} that $\twonorm{\bfR- \bfGamma}= O(1)=\twonorm{\bfS- \bfSigma}\,\,\,  \as$
Estimators $\widehat \bfR$, $\widehat \bfS$ based on $\bfR$ and $\bfS$, respectively, such that as $p\to \infty$,
\begin{equation}\label{eq:setew1}
\twonorm{\widehat\bfR- \bfGamma}= o(1)=\twonorm{\widehat\bfS- \bfSigma}\, \quad \as
\end{equation}
are more desirable. So how can we construct them? The authors of \cite{bickel:levina:2008a,bickel:levina:2008b} considered estimators of the form
\begin{equation}\label{eq:estimators}
\widehat \bfS_{ij}=\big(S_{ij} \1(|S_{ij}|>t_p) \big) \quad \text{ and } \quad \widehat \bfR_{ij}=\big( R_{ij} \1(|R_{ij}|>t_p)\big)\,,
\end{equation}
for some threshold sequence $t_p\to 0$. For $t_p=c\sqrt{(\log p) /n}$, \cite[Theorem~1]{bickel:levina:2008a} shows \eqref{eq:setew1} under some technical conditions on $\bfA$ and the assumption that the iid noise $(Z_{ij})$ is standard Gaussian. 
Gaussianity is a very strong assumption. We will only require a finite sixth moment.

\begin{theorem}\label{thm:pspsp}
We consider the iid case $\X=\bfZ$. Assume $p/n\to\gamma>0$ and $\E[Z^6]< \infty$. Set 
\begin{equation*}
t_p= M \sqrt{\frac{\log p}{n}}\,,\quad \text{ for some } M>2\,.
\end{equation*}
 Then the estimators $\widehat \bfR$, $\widehat \bfS$ defined in \eqref{eq:estimators},
satisfy \eqref{eq:setew1}.
\end{theorem}
\begin{proof}
In view of Theorem~\ref{thm:comp}, it suffices to prove 
\begin{equation*}
\twonorm{\widehat \bfR- \bfI}= o(1)\, \quad \as
\end{equation*}
Define the random variable 
\begin{equation*}
\Delta_p=\1\Big(\max_{1\le i\neq j\le p} |R_{ij}|>t_p\Big)\,.
\end{equation*}
Since $\E[Z^6]<\infty $ the results in \cite{li:liu:rosalsky:2010} imply
\begin{equation*}
\sqrt{\frac{n}{\log p}} \, \max_{1\le i\neq j\le p} |R_{ij}|\cas 2\,, \qquad p\to \infty;
\end{equation*}
see also \cite{jiang:2004b,li:rosalsky:2006,li:qi:rosalsky:2012} for the fluctuations of the largest off-diagonal
entry of $\bfR$. This implies $\Delta_p=0$ a.s. for large $p$.

We have, for $\pto$,
\begin{equation*}
\twonorm{\widehat \bfR- \bfI} \le \twonorm{\diag(\widehat \bfR)-\bfI}+ \twonorm{\offdiag(\widehat \bfR) \Delta_p}+\twonorm{\offdiag(\widehat \bfR) (1-\Delta_p)}=o(1)\quad \as
\end{equation*}
The first summand converges \as~to zero by Theorem~\ref{thm:main}; the second due to the boundedness of $\twonorm{\offdiag(\widehat \bfR)}$ and $\Delta_p \cas 0$; and the third one is identically zero by construction.
\end{proof}
The assumption $\E[Z^6]<\infty$ allows a simple proof but is not necessary. For a weaker condition we refer to Theorems 2.3 and 2.4 in \cite{li:liu:rosalsky:2010}.

In the general case, i.e. $\bfA\neq \bfI$, the estimators $\widehat \bfR$ and $\widehat \bfS$, with $t_p=M\sqrt{(\log p) /n}$ and the constant $M$ depending on a bound of $\twonorm{\bfA}$, yield good approximations as well. Indeed, since
\begin{equation*}
\begin{split}
\bfS= \bfA \diag\big(\tfrac{1}{n} \z\z') \bfA'+\bfA \offdiag\big(\tfrac{1}{n} \z\z') \bfA'
\end{split}
\end{equation*}
and 
\begin{equation*}
\twonorm{\bfA \diag\big(\tfrac{1}{n} \z\z') \bfA'- \bfSigma}\cas 0\,,
\end{equation*}
one just needs certain technical assumptions on $\bfA$ (similar to those in \cite{bickel:levina:2008a,bickel:levina:2008b}) to ensure that the thresholded version of $\bfA \offdiag\big(\tfrac{1}{n} \z\z') \bfA'$ converges to zero in spectral norm. We omit details.

\subsection{Estimating the population eigenvalues}

In this subsection, we propose a procedure to estimate
\begin{equation*}
\spec(\bfGamma):=(\la_1(\bfGamma),\la_2(\bfGamma),\ldots, \la_p(\bfGamma)) \,
\end{equation*}
given some time series observations $(\x_1,\ldots,\x_n)=\X$ from the model $\X=\bfA \z$.
For a general introduction to the topic of spectrum reconstruction the interested reader is referred to \cite[Section~1]{kong:valiant:2017}.

In the iid case, i.e.~$X_{it}=Z_{it}$, we have seen in \eqref{eq:extrcorr0} that if $p/n\to \gamma\in(0,1)$, one has
\begin{equation*}
\lim_{p\to \infty} \la_1(\bfR)= (1+\sqrt{\gamma})^2>\la_1(\bfGamma) =1=\la_p(\bfGamma)>(1-\sqrt{\gamma})^2=\lim_{p\to \infty} \la_p(\bfR)\, \quad \as
\end{equation*}
This means that in high dimensions the eigenvalues of $\bfR$ are not good estimators for the eigenvalues of $\bfGamma$. Our goal is to obtain an accurate approximation of the vector $\spec(\bfGamma)$.

To this end, we need some notation. For an $n\times n$ matrix $\bfM=(M_{ij})$ and $\sigma=(\sigma_1,\ldots, \sigma_k) \in \{1,\ldots, n\}^k$, $k\ge 1$, let 
\begin{equation*}
\bfM^{(\sigma)}= \prod_{i=1}^k M_{\sigma_i,\sigma_{i+1}}\,,
\end{equation*}
where $\sigma_{k+1}$ is interpreted as $\sigma_1$. Recall that for $k\ge 1$,
\begin{equation*}
\tr(\bfSigma^k)=\sum_{i=1}^p (\la_i(\bfSigma))^k\quad \text{ and } \quad \tr(\bfGamma^k)=\sum_{i=1}^p (\la_i(\bfGamma))^k\,.
\end{equation*}
The following result \cite[Fact 2]{kong:valiant:2017} is useful to estimate population covariance eigenvalues.
\begin{lemma}\label{lem:estev}
We consider the matrix $\X=\bfA \bfZ$ in \eqref{eq:model}. Let $k\ge 1$ and $\sigma_1,\ldots, \sigma_k \in \{1,\ldots, n\}$ be pairwise distinct. Then we have
\begin{equation}\label{eq:traceC}
 \E[(\X'\X)^{(\sigma)}]= \tr(\bfSigma^k)\,.
\end{equation}
\end{lemma}

\begin{proof}
Recall that $\bfSigma=\bfA\bfA'$ and 
\begin{equation*}
X_{it}=\sum_{k=1}^p A_{ik} Z_{kt} \quad \text{ and } \quad (\X'\X)_{ij}=\sum_{t=1}^p X_{ti}X_{tj}\,.
\end{equation*}
We have
\begin{equation*}
\begin{split}
 \E[(\X'\X)^{(\sigma)}]&=  \E\Big[ \prod_{i=1}^k \sum_{t_i=1}^p X_{t_i,\sigma_i} X_{t_i,\sigma_{i+1}}  \Big]=
 \sum_{t_1,\ldots,t_k=1}^p \E\Big[ \prod_{i=1}^k X_{t_i,\sigma_i} X_{t_{i-1},\sigma_i} \Big] \\
&= \sum_{t_1,\ldots,t_k=1}^p  \prod_{i=1}^k \E\big[X_{t_i,\sigma_i} X_{t_{i-1},\sigma_i} \big] = \sum_{t_1,\ldots,t_k=1}^p  \prod_{i=1}^k \bfSigma_{t_i,t_{i-1}}\\
&= \tr(\bfSigma^k)\,,
\end{split}
\end{equation*}
where $t_0$ was interpreted as $t_k$.
\end{proof}

For $k\ge 1$ we call $\sigma=(\sigma_1,\ldots,\sigma_k)\in \{ 1,\ldots, n\}^k$ a $k$-path.
While $(\X'\X)^{(\sigma)}$ is an unbiased estimator for $\tr(\bfSigma^k)$, its variance is quite large. The natural way to reduce the variance would be to average over all $k$-paths with distinct entries. However, such an implementation comes at a high computational price. If we instead average over all $k$-paths, we obtain $\tr(\bfS^k)$, which is easy to compute but biased in high dimensions. 
In \cite{kong:valiant:2017}, a ``theoretically optimal and computationally efficient'' algorithm to overcome this issue is studied. They suggest to average over all $\binom{n}{k}$ increasing $k$-paths; i.e. $\sigma_1<\sigma_2<\cdots<\sigma_k$; and consider
\begin{equation}\label{eq:alg0}
\binom{n}{k}^{-1} \sum_{\sigma \text{ increasing } k\text{-path}} (\X'\X)^{(\sigma)}
\end{equation}
By \cite[Lemma~1]{kong:valiant:2017}, the expression in \eqref{eq:alg0} can be written as 
\begin{equation}\label{eq:alg}
\binom{n}{k}^{-1}\tr\big( \mathbf{G}^{k-1} \X'\X \big)\,,
\end{equation}
where $\mathbf{G}$ denotes the matrix $\X'\X$ with the diagonal and lower triangular entries set to zero. 
\par

Based on Theorem~\ref{thm:main} and \eqref{eq:alg} we propose a 3-step-procedure to estimate $\spec(\bfGamma)$ from the data matrix $\X$. \medskip

{\it Step 1.} Set $\bfB:=\diag(\bfS)$ and consider the modified sample covariance matrix
\begin{equation*}
\tfrac{1}{n} \bfB^{-1/2} \X \X' \bfB^{-1/2}\,.
\end{equation*}
If we interpret $\bfB$ as deterministic for a second, the associated population covariance matrix would be $\bfB^{-1/2} \bfSigma \bfB^{-1/2}$. 
For different values of $k\ge 2$, we estimate
\begin{equation*}
m_k = \tr\Big(\big(\bfB^{-1/2} \bfSigma \bfB^{-1/2}\big)^k \Big) 
\end{equation*}
 via \eqref{eq:alg} and set
\begin{equation*}
\widehat m_k = \binom{n}{k}^{-1}\tr\big( \mathbf{G}^{k-1}  \X \bfB^{-1} \X' \big)\,,
\end{equation*}
where $\mathbf{G}$ denotes the matrix $\X \bfB^{-1} \X'$ with the diagonal and lower triangular entries set to zero. Some properties of this estimator are discussed in \cite{kong:valiant:2017}.

{\it Step 2.} By Theorem~\ref{thm:main}, $\diag(\bfSigma)$ can be approximated by $\bfB$. Indeed, $\bfB$ concentrates closely around $\diag(\bfSigma)$. This implies that for $\ell>0$,
\begin{equation}\label{eq:estimate}
\big(p, \widehat m_2, \widehat m_3, \ldots, \widehat m_{\ell}\big) 
\end{equation}
estimates
\begin{equation*}
\Big(\sum_{i=1}^p \la_i(\bfGamma), \sum_{i=1}^p (\la_i(\bfGamma))^2, \ldots, \sum_{i=1}^p (\la_i(\bfGamma))^{\ell}\Big) \,.
\end{equation*}

{\it Step 3.} Based on the (estimated) moments in \eqref{eq:estimate}, we can finally estimate
\begin{equation*}
(\la_1(\bfGamma),\la_2(\bfGamma),\ldots, \la_p(\bfGamma)) \,.
\end{equation*}
The introduction of \cite{kong:valiant:2017} contains an overview of various approaches to the spectrum estimation given moments. For an implementation of the estimation of $\spec(\bfGamma)$ given \eqref{eq:estimate}, and $L_1$ error bounds we refer to Section 3 in \cite{kong:valiant:2017}. 

\begin{remark}{\em
The paper \cite{kleiber:stoyanov:2013} provides a detailed overview about the characterization of probability distributions through their moments. It discusses Carleman's condition, which is widely used in random matrix theory, and many more necessary and sufficient criteria.  
}\end{remark}

\subsection{A generalized spiked population correlation model}

The spiked population covariance model was introduced by Johnstone \cite{johnstone:2001} in 2001. In its base form, all the eigenvalues of the population covariance matrix are one, except for a fixed number of so-called spike eigenvalues. The motivation behind the spiked  population covariance model was to provide a better fit to time series in finance and other areas. Since its birth in 2001, many generalizations of the original spiked population covariance model have been proposed, and the effects of the spikes on the sample eigenvalues have been studied; see \cite{debashis:aue:2014,paul:2007,bai:yao:2012} and the references therein. 

In the (generalized) spiked population correlation model the population correlation matrix $\bfGamma$ has the blockdiagonal structure
\begin{equation}\label{eq:spikedR}
\bfGamma_p = \begin{pmatrix}
\bfLambda & \bf0 \\
\bf0 & \bfV_p
\end{pmatrix}\, \in \R^{p\times p}\,.
\end{equation}
We assume that
\begin{itemize}
	\item[(A1)] $\bfLambda$ is a positive semidefinite $m\times m$ matrix for some fixed $m>0$ and $\diag(\bfLambda)=\bfI_m$. The eigenvalues of $\bfLambda$ in decreasing order are  
\begin{equation*}
\underbrace{\alpha_1,\ldots,\alpha_1}_{m_1}, \underbrace{\alpha_2,\ldots,\alpha_2}_{m_2}, \ldots, \underbrace{\alpha_K,\ldots,\alpha_K}_{m_K}\,,
\end{equation*}
where the multiplicities $(m_i)$ of the eigenvalues satisfy $m_1+\cdots+m_K=m$.
\item[(A2)] The empirical spectral distribution of $\bfV_p$ converges, with probability one, weakly to a probability distribution $H$ as $p\to \infty$.  
\item[(A3)] At least one $\alpha_i$ does not lie in $\supp(H)$, the support of $H$. 
\item[(A4)] We require
\begin{equation*}
 \lim_{\pto} \max_{j=1,\ldots, p-m} d(\beta_{p,j},\supp(H)) =0\,,
\end{equation*}
where $\beta_{p,1},\ldots,\beta_{p,p-m}$ are the eigenvalues of $\bfV_p$ and $d(x,A)$ denotes the Euclidean distance of the point $x$ from the set $A$. 
\end{itemize}

In words, the eigenvalues of $\bfV_p$ lie in $\supp(H)$. Since the spectra of $\bfV_p$ and $\bfGamma_p$ differ by exactly $m$ values, $\bfV_p$ and $\bfGamma_p$ possess the same limiting spectral distribution $H$. Eigenvalues $\alpha_i\notin \supp(H)$ are called spike eigenvalues or simply spikes. By construction, a spike $\alpha_i$ is an eigenvalue of $\bfGamma_p$ with multiplicity $m_i$ for all $p$ sufficiently large. 

The eigenvalues of $\bfGamma_p$ are
\begin{equation*}
\underbrace{\alpha_1,\ldots,\alpha_1}_{m_1}, \underbrace{\alpha_2,\ldots,\alpha_2}_{m_2}, \ldots, \underbrace{\alpha_K,\ldots,\alpha_K}_{m_K}, \beta_{p,1}, \ldots,\beta_{p,p-m}\,.
\end{equation*}
We denote their ordered values by $\delta_1\ge \delta_2\ge \cdots \ge \delta_p$. For a spike eigenvalue $\alpha_i$ let 
\begin{equation}\label{eq:ranks}
\nu_i+1:= \min\{1\le \ell\le p\,:\, \delta_\ell=\alpha_i\}\,.
\end{equation}
In other words, there are $\nu_i$ eigenvalues of $\bfGamma_p$ larger than $\alpha_i$ and $p-\nu_i-m_i$ smaller ones. 

For $\alpha \in (\supp(H))^c$ we define the function
\begin{equation}\label{eq:psi}
\psi(\alpha)=\psi_{\gamma,H}(\alpha)=\alpha+\gamma \int \frac{t \alpha}{\alpha-t} \dint H(t)\,.
\end{equation}
Some properties of $\psi$ are discussed in \cite{bai:yao:2012}. For our purpose it is only important to know that $\psi$ is indeed well defined. Its derivative is 
\begin{equation}\label{eq:psiderivative}
\psi'(\alpha)=1-\gamma \int \frac{t^2}{(\alpha-t)^2} \dint H(t)\,.
\end{equation}
The next theorem explains how spikes of the population correlation matrix $\bfGamma$ influence the spectrum of the sample correlation matrix $\bfR$.  

\begin{theorem}\label{thm:spikes}
We consider the model $\X=\bfA\bfZ$ from \eqref{eq:model}, where the matrix $\bfA$ satisfies \eqref{assumptionA} and $\z$ has iid entries. Let $p/n \to \gamma>0$, $\E[Z^4]< \infty$ and assume (A1)-(A4). 
\begin{itemize}
\item
For a spike eigenvalue $\alpha_i$ of multiplicity $m_i$ satisfying $\psi'(\alpha_i)>0$, we have
\begin{equation*}
\lim_{\pto} \la_{\nu_i+\ell}(\bfR) = \psi(\alpha_i)\quad \as\,,\quad 1\le \ell\le m_i\,.
\end{equation*}
\item
For a spike eigenvalue $\alpha_i$ of multiplicity $m_i$ satisfying $\psi'(\alpha_i)\le 0$, we have
\begin{equation*}
\lim_{\pto} \la_{\nu_i+\ell}(\bfR) = F_{\gamma,H}^{-1}(H(\alpha_i))
\quad \as\,,\quad 1\le \ell\le m_i\,,
\end{equation*}
where $F_{\gamma,H}$ is the limiting spectral distribution of $\bfR$; see Theorem~\ref{thm:lsdR} part (1); and $F_{\gamma,H}^{-1}(H(\alpha_i))$ is the $H(\alpha_i)$-quantile of $F_{\gamma,H}$.
\end{itemize}
\end{theorem}

\begin{proof}
By \eqref{eq:compeig1}, the statements of the theorem follow immediately from Theorems 4.1 and 4.2 in \cite{bai:yao:2012}; compare also with Theorem 11.3 in \cite{yao:zheng:bai:2015}.
\end{proof}

\begin{example}{\em
We consider a special case of Theorem \ref{thm:spikes}. Let $\gamma\in (0,1]$, $m=2$ and the eigenvalues of $\bfLambda$ be $(\alpha_1,\alpha_2)=(1+\delta,1-\delta)$ for some $\delta\in [0,1]$.
Choose $\bfV_p=\bfI_{p-m}$, the $(p-m)$-dimensional identity matrix. Then the limiting spectral distribution $H$ is the Dirac measure at $1$ and we have for $\alpha\neq 1$,
\begin{equation*}
\psi(\alpha)=\alpha+\frac{\gamma \alpha}{\alpha-1} \quad \text{and} \quad \psi'(\alpha)=1-\frac{\gamma}{(\alpha-1)^2}\,.
\end{equation*}
A simple calculation shows that $\psi'(\alpha)>0$ if and only if $\alpha>1+\sqrt{\gamma}$ or $\alpha<1-\sqrt{\gamma}$.

Therefore, if $\delta\le \sqrt{\gamma}$, there are no spike eigenvalues. The extreme eigenvalues of $\bfR$ tend towards the endpoints of the support of the limiting spectral distribution, namely
\begin{equation*}
\la_1(\bfR)\cas (1+\sqrt{\gamma})^2 \quad \text{ and } \quad \la_p(\bfR)\cas (1-\sqrt{\gamma})^2\,,\quad \pto\,.
\end{equation*} 
If $\delta> \sqrt{\gamma}$, then $\alpha_1$ and $\alpha_2$ are spikes with multiplicity $1$ and Theorem~\ref{thm:spikes} yields
\begin{equation*}
\la_1(\bfR)\cas \psi(\alpha_1) >(1+\sqrt{\gamma})^2 \quad \text{ and } \quad 
\la_p(\bfR)\cas \psi(\alpha_2)<(1-\sqrt{\gamma})^2\,,\quad \pto\,.
\end{equation*} 
In this basic example $\delta$ quantifies the deviation of $\spec(\bfLambda)$ from the support of the limiting spectral distribution $H$. The limits of the extreme sample eigenvalues are only pulled out of the support of the limiting spectral distribution of $\bfR$ if $\delta$ exceeds a certain threshold. 
}\end{example}

\section{Proofs}\label{sec:proofs}\setcounter{equation}{0}

\subsection{Proof of Theorem~\ref{thm:main}}\label{sec:proofofmainthm}

We start with the proof of \eqref{eq:main}. 
Recall that 
\begin{equation}\label{eq:model34}
\X=\bfA \z \quad \text{ and } \quad \bfS= \frac{1}{n} \X\X'\,,
\end{equation}
where $\z$ is a $p\times n$ matrix of iid random variables with generic entry $Z$ such that $\E[Z]=0,\E[Z^2]=1$ and $\E[Z^4]<\infty$.

We have 
\begin{equation*}
\sqrt{n/p} \,\twonorm{\diag(\bfS)-\diag(\bfA \bfA')} =\sqrt{n/p} \,\max_{i=1,\ldots,p}  \Big|\frac{1}{n} \sum_{t=1}^n \Big(\sum_{j=1}^p A_{ij} Z_{jt}\Big)^2 - \sum_{j=1}^p A_{ij}^2 \Big|\,.
\end{equation*}
Since $c^{-1} \le \min_i \sum_{j=1}^p A_{ij}^2\le \max_i \sum_{j=1}^p A_{ij}^2 \le \twonorm{\bfA}^2\le c$ for some constant $c>0$,  
one can assume without loss of generality that $\sum_{j=1}^p A_{ij}^2=1$. Otherwise simply divide the $i$th diagonal entry by $\sum_{j=1}^p A_{ij}^2$. Thus, we need to show
\begin{equation*}
\sqrt{n/p} \,\twonorm{\diag(\bfS)-\bfI} \cas 0\,, \quad p \to \infty\,.
\end{equation*} 

The proof will be in 3 steps.
\begin{enumerate}
\item Truncation: Define the truncated random variables $\widehat Z_{ij}= Z_{ij} \1(|Z_{ij}|\le \delta_p (np)^{1/4})$ for a suitable sequence $\delta_p\to 0$ and construct $\widehat \X$ and $\widehat \bfS$ analogously to \eqref{eq:model34}. We will show that $\sqrt{n/p}\, \twonorm{\diag(\bfS)-\diag(\widehat \bfS)} \cas 0$ as $p\to \infty$.
\item Renormalization: Set $\wt Z_{ij}= \frac{\widehat Z_{ij}-\E[\widehat Z_{ij}]}{\sqrt{\Var(\widehat Z_{ij})}}$. For the matrix $\wt \bfS$, defined analogously to \eqref{eq:model34}, we then show that $\sqrt{n/p}\, \twonorm{\diag(\wt \bfS)-\diag(\widehat \bfS)} \cas 0$ as $p\to \infty$.
\item We prove that $\sqrt{n/p}\, \twonorm{\diag(\wt \bfS)-\bfI} \cas 0$ as $p\to \infty$.
\end{enumerate}

\subsubsection*{Step (1)}

For $i,j\ge 1$ let $\widehat Z_{ij}= Z_{ij} \1(|Z_{ij}|\le \delta_p (np)^{1/4})$, where the sequence of positive $\delta_p$ satisfies 
\begin{equation}\label{eq:deltap}
\lim_{p\to \infty} \delta_p=0\,, \quad \lim_{p\to \infty}  \delta_p^{-4} \E[|Z|^4\1(|Z|>\delta_p (np)^{1/4})]=0\,, \quad \delta_p (np)^{1/4}\to \infty\,.
\end{equation}
We refer to \cite[p.~1408]{chen:pan:2012} for the construction of $\delta_p$. In what follows, we will often drop the indices to simplify notation. 

For the matrix $\widehat \z=(\widehat Z_{ij})_{i\le p; j\le n}$ it is shown in \cite[p.~1409]{chen:pan:2012} that 
\begin{equation*}
\P\big(\limsup_{p\to \infty}\, \{ \z \neq \widehat \z\}) =0\,,
\end{equation*}
which implies $\sqrt{n/p}\, \twonorm{\diag(\bfS)-\diag(\widehat \bfS)} \cas 0$ as $p\to \infty$, where $\widehat \bfS=\bfA \widehat \z \widehat \z' \bfA'$.

\subsubsection*{Step (2)}

We introduce the matrix $\wt \z=(\wt Z_{ij})_{i\le p, j\le n}$ with entries $\wt Z_{ij}= \frac{\widehat Z_{ij}-\E[\widehat Z_{ij}]}{\sqrt{\Var(\widehat Z_{ij})}}$ and define $\wt \bfS$ analogously to \eqref{eq:model34}, i.e., $\wt \bfS=\bfA \wt \z \wt \z' \bfA'$. Then we have
\begin{equation*}
\begin{split}
\twonorm{\diag(\widehat \bfS)-\diag(\wt \bfS)}&= \max_{i\le p} |\widehat S_{ii} -\wt S_{ii}| \le \twonorm{\widehat \bfS-\wt \bfS}\\
&= \frac{1}{n} \twonorm{\bfA (\widehat \z \widehat \z' -\wt \z \wt \z' ) \bfA'} \le \twonorm{\bfA}^2 \frac{1}{n} \twonorm{\widehat \z \widehat \z' -\wt \z \wt \z' }\,.
\end{split}
\end{equation*}
Since $(\twonorm{\bfA})$ is uniformly bounded, it is enough to prove that
\begin{equation}\label{eq:setew}
(np)^{-1/2}\, \twonorm{\widehat \z \widehat \z' -\wt \z \wt \z' }\cas 0\,, \quad p\to \infty\,.
\end{equation}
By construction, we have $\E [\widehat Z_{11}] \to 0$ and $\Var(\widehat Z_{11})\to 1$. More precisely, one gets
\begin{equation}\label{eq:rate121}
|\E [\widehat Z_{11}]| = |\E[Z_{11} \1(|Z_{11}|> \delta (np)^{1/4})]|\le 
\E\Big[|Z_{11}| \Big|\frac{Z_{11}}{\delta (np)^{1/4}}\Big|^3 \1(|Z_{11}|> \delta (np)^{1/4})\Big]= o((np)^{-3/4})
\end{equation}
and 
\begin{equation}\label{eq:rate122}
|\Var(\widehat Z_{11})-1|=\big| \E[ Z_{11}^2 \1(|Z_{11}|> \delta (np)^{1/4}) ] +o((np)^{-3/2})\big|= o((np)^{-1/2})\,.
\end{equation}
In view of the boundedness of the largest eigenvalue of $n^{-1} \widehat \z \widehat \z'$ and \eqref{eq:rate121}, we have \as
\begin{equation}\label{eq:hrsgersg}
\twonorm{\widehat \z \widehat \z'}=O(n) \quad \text{ and } \quad \twonorm{\E[\widehat \z]}^2 =np |\E [\widehat Z_{11}]|^2 = o((np)^{-1/2})\,.
\end{equation}
By multiplying out one obtains
\begin{equation*}
\begin{split}
\frac{\widehat \z \widehat \z' -\wt \z \wt \z' }{(np)^{1/2}}&= \frac{(\Var(\widehat Z_{11})-1)\widehat \z \widehat \z'}{\Var(\widehat Z_{11})(np)^{1/2}}+
\frac{\widehat \z \E[\widehat \z'] + \E[\widehat \z] \widehat \z' - \E[\widehat \z]\E[\widehat \z']}{\Var(\widehat Z_{11})(np)^{1/2}}\,
\end{split}
\end{equation*}
and therefore we conclude by \eqref{eq:rate122} and \eqref{eq:hrsgersg} that
\begin{equation*}
\begin{split}
\frac{\twonorm{\widehat \z \widehat \z' -\wt \z \wt \z' }}{(np)^{1/2}}&\le  \frac{c\, |\Var(\widehat Z_{11})-1| \twonorm{\widehat \z \widehat \z'}}{(np)^{1/2}}+
\frac{c\,\twonorm{\E[\widehat \z]\E[\widehat \z']}}{(np)^{1/2}}+
\frac{c\, \twonorm{\widehat \z} \twonorm{\E[\widehat \z]}}{(np)^{1/2}}\\
&= o(p^{-1}) + o(n^{-1}p^{-1}) + o(n^{-1/4}p^{-3/4})\, \quad \as 
\end{split}
\end{equation*}

\subsubsection*{Step (3)}

By steps (1) and (2), it is sufficient to work with the truncated and renormalized variables $\wt Z_{ij}$ and the matrix $\wt \bfS$. For simplicity of notation we will omit the tilde in the rest of this proof.  Hence, in addition to $E[Z]=0, \E[Z^2]=1$ and $\E[Z^4]<\infty$ we assume $|Z|\le \delta_p (np)^{1/4}$, 
where the sequence of positive $\delta_p$ satisfies $\delta_p\to 0$ and $\delta_p (np)^{1/4}\to \infty$ as $p \to \infty$.

Our goal is to prove that 
\begin{equation}\label{eq:aega11}
\sqrt{n/p}\, \twonorm{\diag( \bfS)-\bfI} = \sqrt{n/p} \, \max_{i=1,\ldots,p}  | S_{ii}-1 |\cas 0\,,\quad  p\to \infty\,.
\end{equation}
For $1\le i\le p$ we have
\begin{equation}\label{eq:exps}
S_{ii}= \frac{1}{n} \sum_{j=1}^p A_{ij}^2 \sum_{t=1}^n Z_{jt}^2 + \frac{1}{n}\sum_{t=1}^n \sum_{j_1\neq j_2=1}^p A_{ij_1}A_{ij_2} Z_{j_1 t}Z_{j_2 t}=: S_i(1)+S_i(2)\,.
\end{equation}
Due to $\diag(\A\A')=\bfI$, we observe that $\E[S_{ii}]=1$. Clearly,
\begin{equation}\label{eq:aega}
\max_{i=1,\ldots,p}  \sqrt{n/p} \, | S_{ii}-1 |\le
\max_{i=1,\ldots,p}  \sqrt{n/p} \, | S_i(1)-1 | + \max_{i=1,\ldots,p}  \sqrt{n/p}\, | S_i(2) |\,.
\end{equation}
Writing $D_j= n^{-1} \sum_{t=1}^n Z_{jt}^2$ one sees that
\begin{equation}\label{eq:segesa}
\begin{split}
\max_{i=1,\ldots,p}  \sqrt{n/p} \,  | S_i(1)-1 |&\le
\max_{i=1,\ldots,p}  \sqrt{n/p} \, \sum_{j=1}^p A_{ij}^2 | D_j-1 |\\
&\le \sqrt{n/p} \,  \max_{i=1,\ldots,p}  | D_i-1 | \cas 0\,,
\end{split}
\end{equation}
by Lemma \ref{lem:2}. For the second summand in \eqref{eq:aega}, we will show that for any $\vep>0$, $\ell>0$,
\begin{equation}\label{eq:aefe}
\P\Big(\max_{i=1,\ldots,p}  \sqrt{n/p} \,  | S_i(2) | >\vep \Big)=o (p^{-\ell})\,, 
\end{equation}
from which $\max_{i=1,\ldots,p}  \sqrt{n/p} \,  | S_i(2) | \cas 0$ follows via the first Borel--Cantelli lemma. In view of \eqref{eq:segesa}, the proof of \eqref{eq:aega11} is complete. 
\par

It remains to show \eqref{eq:aefe}. Let 
\begin{equation*}
Y_{it}:=\sum_{j_1\neq j_2=1}^p A_{ij_1}A_{ij_2} Z_{j_1 t}Z_{j_2 t}
\end{equation*}
so that $S_i(2)= \frac{1}{n} \sum_{t=1}^n Y_{it}$.
First, we derive moment inequalities for $Y_{it}$. Clearly $\E[Y_{it}]=0$.
Since $|Z|\le \delta (np)^{1/4}$ we have for $q\ge 2$ that
\begin{equation*}
(\E[|Z|^q])^{1/q}\le \big(\E[Z^2] (\delta (np)^{1/4})^{q-2}\big)^{1/q}=(\delta (np)^{1/4})^{1-2/q}\,.
\end{equation*}
We notice that
\begin{equation*}
\begin{split}
\sum_{j_1\neq j_2=1}^p  A_{ij_1}^2A_{ij_2}^2 \le  \Big(\sum_{j=1}^p A_{ij}^2\Big)^2 =1\,.
\end{split}
\end{equation*}
Therefore an application of Lemma~\ref{lem:qmoment} yields
\begin{equation}\label{eq:momentY}
\E[ |Y_{it}|^q ]\le C^q \, q^q\,  (\delta^2 \sqrt{np})^{q-2}\,,\quad q=2,3,\ldots\,,
\end{equation}
where the positive constant $C$ does not depend on $q$.

For an appropriate sequence $\delta\to 0$ with $\delta^4 np \to \infty$ we choose an integer sequence $h=h_p\to \infty$ such that as $p\to \infty$,
\begin{equation}\label{eq:seqh}
\frac{h}{\log p}\to \infty\,,\quad \frac{h^2\delta^2}{\log(p \delta^4)}\to 0 \quad \text{ and } \quad \frac{h}{\log(p \delta^4)}>1\,.
\end{equation}
Following \cite[p.~1412]{chen:pan:2012} we have for $\vep>0$,
\begin{equation*}
\begin{split}
\P&\Big(\max_{i=1,\ldots,p}  \sqrt{n/p} \,  | S_i(2) | >\vep \Big)=
\P\Big(\max_{i=1,\ldots,p}  (np)^{-1/2} \,  \Big| \sum_{t=1}^n Y_{it} \Big| >\vep \Big)\\
&\le \sum_{i=1}^p \P\Big(  (np)^{-1/2} \,  \Big| \sum_{t=1}^n Y_{it} \Big| >\vep \Big)\\
&\le \sum_{i=1}^p \vep^{-h}(np)^{-h/2} \E\Big[\Big| \sum_{t=1}^n Y_{it} \Big|^h \Big]\\
&\stackrel{(1)}{\le} \sum_{i=1}^p \vep^{-h}(np)^{-h/2}
\sum_{m=1}^{h/2} \sum_{1\le t_1<\cdots<t_m\le n}
\mathop{\sum_{i_1+\cdots +i_m=h }}_{i_j \ge 2}  \binom{h}{i_1,\ldots, i_m} \E[ |Y_{it_1}|^{i_1} ]\E[ |Y_{it_2}|^{i_2} ]\cdots \E[ |Y_{it_m}|^{i_m} ]\\
&\stackrel{(2)}{\le} p \,\vep^{-h}(np)^{-h/2}
\sum_{m=1}^{h/2} 
  n^m \mathop{\sum_{i_1+\cdots +i_m=h }}_{i_j \ge 2} \binom{h}{i_1,\ldots, i_m} C^h h^h (\delta^2 \sqrt{np})^{h-2m}\\
&\stackrel{(3)}{\le} p \,\vep^{-h} (C h \delta^2)^h
\sum_{m=1}^{h/2} m^h (p \delta^4)^{-m}
\stackrel{(4)}{\le} p \,\vep^{-h} (C h \delta^2)^h
\frac{h}{2} \Big(\frac{h}{\log(p \delta^4)}\Big)^h\\
&= \Big( \Big( \frac{ph}{2} \Big)^{1/h} \frac{Ch^2\delta^2}{\vep \log(p \delta^4)} \Big)^h \stackrel{(5)}{=}o(p^{-\ell})
\end{split}
\end{equation*}
for any $\ell>0$. 

Below are some additional explanations of the inequalities:
\begin{enumerate}
\item Multinomial theorem and $\E[Y_{it}]=0$.
\item We used \eqref{eq:momentY}, $\sum_{1\le t_1<\cdots<t_m\le n} 1= \binom{n}{m}\le n^m$ and $\prod i_j^{i_j}\le h^h$.     
\item Using $\sum_{i_1+\cdots +i_m=h; i_j \ge 2} \binom{h}{i_1,\ldots, i_m}\le m^h$ and simplifying.   
\item We apply the elementary inequality
\begin{equation*}
a^{-t}t^b \le \Big(\frac{b}{\log a}\Big)^b\,,\quad \text{ for } a>1,b>0,t\ge1 \text{ and } \frac{b}{\log a}>1\,.
\end{equation*}  
\item In view of \eqref{eq:seqh}, $(ph/2)^{1/h}$ converges to $1$ and we have for sufficiently large $p$ that
\begin{equation*}
\Big( \frac{ph}{2} \Big)^{1/h} \frac{Ch^2\delta^2}{\vep \log(p \delta^4)}<\xi\,,
\end{equation*}
for some $\xi \in (0,1)$. Note that $\xi^h=o(p^{-\ell})$ for any $\ell>0$.               
\end{enumerate}

This finishes the proof of \eqref{eq:aefe}. 

Finally, we prove \eqref{eq:mainsqrt}. 
In view of the inequality 
\begin{equation*}
|a^{-1/2} -b^{-1/2}|= \frac{|a-b|}{|ab||a^{-1/2} +b^{-1/2}|}\le c |a-b|\,,
\end{equation*}
for $a,b$ in some interval bounded away from zero and $\infty$, equation \eqref{eq:mainsqrt} follows from \eqref{eq:main}.
\hfill \qed

\subsection{Proof of Theorem~\ref{thm:exiid}}
In view of \eqref{eq:compeig1}, it is enough to prove \eqref{eq:extrcorr1}.
Since the rank of $\bfS$ is at most $\min(p,n)$ we have $\la_{\min(p,n)+1}(\bfS)=0$. Thus, it makes sense to interpret $\la_{n}(\bfS)$ as the smallest eigenvalue of $\bfS$ if $p\ge n$.

We start with the case $\gamma\in (0,\infty)$. It was shown in \cite{bai:yin:1993} that as $p\to \infty$,
\begin{equation*}
 \la_1(\bfS)\cas (1+\sqrt{\gamma})^2 \quad \text{ and } \quad
\la_{\min(p,n)}(\bfS)\cas (1-\sqrt{\gamma})^2\,,
\end{equation*}
which implies \eqref{eq:extrcorr1}.

In the case $\gamma=0$, Theorem 3 in \cite{chen:pan:2012} asserts that 
\begin{equation}\label{eq:extrcorr}
\sqrt{n/p}\,  (\la_1(\bfS)-1) \cas 2=2 + \sqrt{\gamma} \,.
\end{equation}
The authors of \cite{chen:pan:2012} focused only on the largest eigenvalue. Fortunately, with a small adjustment one can prove the \as~convergence of the smallest eigenvalue in the same way. 
From the theorem in \cite{bai:yin:1988} we get that
\begin{equation*}
\limsup_{p\to \infty} \sqrt{n/p}\,  (\la_p(\bfS)-1) \le -2 \,\quad \as
\end{equation*}
From steps (1) and (2) of the proof of Theorem~\ref{thm:main} we know that truncating and renormalizing the entries of $\X$ does not change the asymptotic \as~behavior of $\la_p(\bfS)$. Therefore, it is sufficient to prove 
\begin{equation}\label{eq:gesgs}
\liminf_{p\to \infty} \sqrt{n/p}\,  (\la_p(\wt \bfS)-1) \ge -2 \,\quad \as,
\end{equation}
where (as in proof of Theorem~\ref{thm:main}) $\wt \bfS$ denotes the sample covariance matrix based on the truncated and renormalized entries.
Note that because of 
\begin{equation*}
\sqrt{n/p}\, \twonorm{\wt \bfS-\bfI} = \sqrt{n/p} \,\max(\la_1(\wt \bfS)-1, -\la_p(\wt \bfS)+1),
\end{equation*}
the inequality
\begin{equation}\label{eq:gesgs2}
\limsup_{p\to \infty} \sqrt{n/p}\,  \twonorm{\wt \bfS-\bfI} \le 2 \,\quad \as
\end{equation}
implies \eqref{eq:gesgs}.

Finally, we prove \eqref{eq:gesgs2}. For any positive integer sequence $k=k_p\to \infty$ and $\vep>0$ we have
\begin{equation*}
\begin{split}
\P\Big( \sqrt{n/p}\,  \twonorm{\wt \bfS-\bfI} >2+\vep\Big)
&\le (2+\vep)^{-2k} \Big(\frac{n}{p} \Big)^k \E\big[ \twonorm{\wt \bfS-\bfI}^{2k} \big]\\
&\le (2+\vep)^{-2k} \Big(\frac{n}{p} \Big)^k \E\Big[ \sum_{i=1}^p \la_i\Big((\wt \bfS-\bfI)^{2k}\Big) \Big]\\
&= (2+\vep)^{-2k} \Big(\frac{n}{p} \Big)^k \E\Big[ \tr \Big((\wt \bfS-\bfI)^{2k}\Big) \Big]
\end{split}
\end{equation*}
With an appropriate choice of the sequence $k$ (see \cite[p.1413]{chen:pan:2012} for details) and some tedious calculations in the spirit of \cite{bai:yin:1988}, it is shown in \cite[pp.~1413-1418]{chen:pan:2012} that 
\begin{equation*}
(2+\vep)^{-2k} \Big(\frac{n}{p} \Big)^k \E\Big[ \tr \Big((\wt \bfS-\bfI)^{2k}\Big) \Big]=o(p^{-\ell})\,, \quad p\to \infty\,,
\end{equation*}
for any $\ell>0$. By the first Borel--Cantelli lemma this implies \eqref{eq:gesgs2}, completing the proof.  \hfill \qed

\appendix

\section{Auxiliary lemmas}\label{sec:appendix}\setcounter{equation}{0}

We state Lemma 2 in \cite{bai:yin:1993}.
\begin{lemma}\label{lem:2}
Let $(X_{ij})$ be a double array of iid random variables and let $\alpha>1/2, \beta\ge 0$ and $M>0$ be constants. Then as $\nto$,
\begin{equation*}
\max_{j\le M n^{\beta}} \Big| n^{-\alpha} \sum_{i=1}^n (X_{ij}-c)  \Big| \cas 0\,,
\end{equation*}
if and only if the following hold:
\begin{equation*}
\E[|X_{11}|^{(1+\beta)/\alpha}] <\infty\,;
\end{equation*}
\begin{eqnarray*}\label{eq:A}
c =
\left\{\begin{array}{ll}
 \E[X_{11}]\,, & \mbox{if } \alpha\le 1,  \\
\text{any number} \,, & \mbox{if } \alpha>1 .
\end{array}\right. 
\end{eqnarray*}
\end{lemma}

The next result is Lemma 7.10 in \cite{erdos:yau:2017}.
\begin{lemma}\label{lem:qmoment}
Let $X_1,\ldots,X_N$ be independent centered random variables and assume that 
\begin{equation*}
(\E[|X_i|^q])^{1/q} \le \mu_q\,,\quad 1\le i\le N; q=2,3,\ldots
\end{equation*} 
for some fixed constants $\mu_q$. Then we have for any deterministic complex numbers $a_{ij}, 1\le i,j\le N$ that 
\begin{equation*}
\Big(\E\Big[ \Big|\sum_{i\neq j=1}^N a_{ij}X_i X_j \Big|^q \Big]\Big)^{1/q}\le C \, q\, \mu_q^2 \Big(\sum_{i\neq j=1}^N |a_{ij}|^2\Big)^{1/2}\,,\quad q=2,3,\ldots,
\end{equation*} 
where the constant $C$ does not depend on $q$.
\end{lemma}

\end{document}